\documentclass{amsart}

\usepackage{amsmath, amssymb}
\usepackage{enumitem}
\input xy
\xyoption{all}
\begin{document}

\newtheorem{theorem}{Theorem}[section]
\newtheorem{proposition}[theorem]{Proposition}
\newtheorem{lemma}[theorem]{Lemma}
\newtheorem{corollary}[theorem]{Corollary}
\newtheorem{fact}[theorem]{Fact}

\theoremstyle{definition}
\newtheorem{definition}[theorem]{Definition}
\newtheorem{conjecture}[theorem]{Conjecture}
\newtheorem{notation}[theorem]{Notation}

\theoremstyle{remark}
\newtheorem{remark}[theorem]{Remark}
\newtheorem{example}[theorem]{Example}
\newtheorem{question}[theorem]{Question}

\numberwithin{equation}{section}

\def\forkindep{\mathrel{\raise0.2ex\hbox{\ooalign{\hidewidth$\vert$\hidewidth\cr\raise-0.9ex\hbox{$\smile$}}}}}
\def\dcl{\mathrm{dcl}}
\def\acl{\mathrm{acl}}
\def\DCF{\mathrm{DCF}_0}
\def\tp{\mathrm{tp}}
\def\stp{\mathrm{stp}}
\def\alg{\mathrm{alg}}

\def\ld{\mathrm{log}\delta}
\def\C{\mathcal C}
\def\U{\mathcal U}
\def\P{\mathcal P}
\def\Q{\mathbb Q}

\title[Constructing types in differentially closed fields]
{Constructing types in differentially closed fields that are analysable in the constants}
\author{Ruizhang Jin}
\date{\today}
\begin{abstract}
  Analysability of finite $U$-rank types are explored both in general and in the theory
  $\DCF$. The well-known fact that the equation $\delta(\ld x)=0$ is analysable in
  but not almost internal to the constants is generalized to show
   that $\underbrace{\ld...\ld}_n x=0$ is not analysable in the constants in $(n-1)$-steps.
   The notion of a \emph{canonical analysis} is introduced -- namely an analysis that is of minimal
   length and interalgebraic with every other analysis of that length. Not every analysable type admits a canonical
   analysis.
   Using
  properties of reductions and coreductions in theories with the canonical base property,
  it is constructed, for any sequence of positive integers $(n_1,...,n_\ell)$,
  a type in $\DCF$ that admits a canonical analysis
   with the property that the $i$th step has $U$-rank $n_i$.
\end{abstract}
\maketitle

\tableofcontents

\section{Introduction}\label{intr}
\noindent
That differential-algebraic geometry is an expansion of algebraic geometry is reflected in model theory by viewing the theory of algebraically closed fields as a reduct of the theory of differentially closed fields. The locus of that reduct is the field of constants. The smallest intermediate reduct that properly expands algebraic geometry is that of differential varieties that are {\em almost internal} to the constants: differential varieties that over possibly additional parameters become definable finite covers of algebraic varieties in the constants. Here already one observes new and interesting geometric and model theoretic phenomena. A further step would be to consider those differential varieties that are built up through a finite sequence of fibrations whose fibres are almost internal to the constants; these are the differential varieties that are {\em analysable in the constants}, and they are the focus of this paper.
In particular, we give some constructions that exhibit the richness of this category.

Differential varieties analysable in the constants have come up recently in applications; it is shown in~\cite{bell2016d} that they give rise to a new class of associative algebras satisfying the classical Dixmier-Moeglin equivalence.

Probably the best known example of an analysable but not internal to the constants differential variety is the one defined by the equation $\displaystyle\delta\left(\frac{\delta x}{x}\right)=0$.
It decomposes as an extension of the additive group of constants by the multiplicative group of constants, without itself being almost internal to the constants.
Our first observation is to generalize this construction by iterating the logarithmic derivative.
Writing $\displaystyle{\ld x:=\frac{\delta x}{x}}$ and $\ld^{(m)}=\underbrace{\ld...\ld}_m$ we consider the equation $\ld^{(m)}x=0$, and show in Section \ref{siter} that while it is analysable in the constants in $m$ steps, it is not analysable in $m-1$ steps.
This is done in Section \ref{siter} by essentially reducing to the $m=2$ case.

Note that  each step in the analysis of $\ld^{(m)}x=0$ is of $U$-rank one.
It is not hard to produce from this example, using methods that work generally in stable theories
 satisfying the canonical base property (CBP), including reductions and coreductions,
  other examples of types analysable in the constants in $m$-steps but not in $m-1$-steps. We may even require this type to satisfy the property that the $i$th step of the analysis by reductions
 of this type is of $U$-rank $n_i$, for any given increasing sequence $\displaystyle(n_i)_{i=1}^m$,
 or that the $i$th step of the analysis by coreductions
 of this type is of $U$-rank $n_i$, for any given decreasing sequence $\displaystyle(n_i)_{i=1}^m$.
 This is done in Section \ref{redu}.

But we look for more; we want analyses of a type $p$ that are {\em canonical} in the sense that up to interalgebraicity there is no other analyses of $p$ in the constants of the same (minimal) length.
Not every finite rank type, even in $\DCF$, admits a canonical analysis (see Example \ref{s4example}).
However, we show in Section \ref{sec5} that
given any sequence of positive integers $(n_1,\dots,n_m)$ there exists in $\DCF$
 a type that has a canonical analysis in the constants with $i$th step having $U$-rank $n_i$.
Unlike in the logarithmic derivative case, these examples are not differential algebraic groups, and hence
that theory is not directly available to us. Our proofs involve a careful algebraic analysis of the equations
that arise.
Note that the situation is very different for differential algebraic groups; in \cite{bell2016d} it is shown that
every differential algebraic group over the constants is analysable in at most 3 steps.

I thank Rahim Moosa, my PhD supervisor, for his advice and inputs during the writing of this paper.

\section{Analysability}\label{sprel}
\noindent
 As a general setting,
 we work in a saturated model $\mathcal U$ of a complete stable theory $T$ that eliminates imaginaries.
 We review in this section some classical notions around finite rank types.
 As a general reference we suggest \cite{pillay1996geometric}. We have provided proofs
 where explicit references were not possible.

Let $\P$ be a set of partial types (over different parameter sets) which is invariant under automorphisms
over $\varnothing$, and $q$ be a stationary type over a parameter set $A$.

Recall that a stationary type $q$ over $A$
 is \emph{$\mathcal P$-internal} (or \emph{almost $\mathcal P$-internal}) if for
some (equivalently any) realization $a$ of $q$, there exists $B \supseteq A$ which is independent from $a$ over $A$,
and $c_1,...,c_k$ realizations
 of types in $\mathcal P$ whose parameter sets are contained in $B$, such that $a\in\dcl(Bc_1...c_k)$ (or $a\in\acl(Bc_1...c_k)$).

 The type $q$ over $A$ is \emph{$\P$-analysable} if for some (equivalently any)
  realization $a$ of $q$, there are $a_1,...,a_k$
  such that $\stp(a_1/A)$ is almost $\P$-internal, $a_{i-1}\in\dcl(Aa_i)$,
$\stp(a_i/Aa_{i-1})$ is almost $\P$-internal for $i=2,3,...,k$, and
$\acl(Aa)=\acl(Aa_k)$. The sequence $(a_i)_{i=1}^k$ mentioned above
is called a \emph{$\mathcal P$-analysis of $q$} and  a \emph{$\mathcal P$-analysis of $a$ over $A$}.
For notational convenience, for any analysis $(a_i)_{i=1}^k$ we use $a_0$ to denote the empty tuple.
We call $k$ the \emph{length} of the analysis. Note that an algebraic type has a $\P$-analysis
of length zero, and an almost $\P$-internal type has a $\P$-analysis of length 1.

The \emph {$U$-type} of the analysis is the sequence $(U(a_i/Aa_{i-1}))_{i=1}^k$.
We say the analysis is \emph{nondegenerated} if each entry of the $U$-type is nonzero.

  Note that the definition of analysable here is in fact the definition of \emph{almost analysable} in the
  literature (for example, section 1 of \cite{moosa2008canonical}), and we may instead say
  that a type is \emph{strictly $\P$-analysable} if $\stp(a_i/a_{i-1})$ is internal (rather than almost internal)
  to $\P$.
  The following proposition proves that these two definitions are in fact equivalent.

\begin{proposition}\label{equiv}
  A stationary type $q$ over $A$ is $\P$-analysable iff
  it is strictly $\P$-analysable.
\end{proposition}

We need the following lemma.

\begin{lemma}\label{lemequiv}
  If a stationary type $q$ over $A$ is almost $\P$-internal, then
  for any $a\vDash q$, there exists a tuple $a_0$ such that
    $\tp(a_0/A)$ is $\P$-internal and
     $\acl(Aa)=\acl(Aa_0)$.
\end{lemma}

\begin{proof}
  Given any realization $a\vDash q$, let $n$ be the least number such that
  there exists an $L_A$-formula $\varphi(x,y,z)$,
  a tuple $b$ independent from $a$ over $A$ and
  a tuple $c$ realizing types in $\P$ such that
  $\vDash\varphi(a,b,c)$ and $\varphi(\mathcal U,b,c)$ is of size $n$.
  We fix these $b$, $c$, and $\varphi$ that satisfy $|\varphi(\mathcal U,b,c)|=n$.

  Step 1. We prove that $\varphi\left(\mathcal U,b,c\right)\subseteq\mathrm{acl}(Aa)$.

  Let $a=a_1,a_2,...,a_n$ be the elements of $\varphi(\mathcal U,b,c)$. Towards a contradiction,
  suppose without loss of generality
  that $a_2\not\in\mathrm{acl}(Aa)$. Then there are $a_2'$, $b'$ and $c'$ such that
  $\tp(a_2'b'c'/Aa)=\tp(a_2bc/Aa)$ and
  $a_2'b'\forkindep_{Aa} a_2...a_nb$. Since $a_2'\not\in\mathrm{acl}(Aa)$ and $a_2'\forkindep_{Aa} a_2...a_nb$,
  $a_2'\not\in\mathrm{acl}(Aaa_2...a_nb)$.
  In particular, $a_2'\neq a_i$ for $i=1,2,...,n$. Also, since $a\forkindep_A b$ and $b\forkindep_{Aa} b'$,
  we have $b\forkindep_A ab'$, and therefore $b\forkindep_{Ab'} a$.
  As $\tp(b'/Aa)=\tp(b/Aa)$ and $b\forkindep_A a$, we have $b'\forkindep_A a$, which,
  together with $b\forkindep_{Ab'} a$, yields $bb'\forkindep_A a$.
  Now the fact that $q$ is almost $\P$-internal is witnessed by
  $a\vDash \varphi(x,b,c)\wedge\varphi(x,b',c')$, and the size of
  $\varphi(\mathcal U,b,c)\wedge\varphi(\mathcal U,b',c')$ is smaller then $n$ (notice that
  $|\varphi(\mathcal U,b,c)|=|\varphi(\mathcal U,b',c')|=n$, but the two sets are not the same), contradicting
  minimality of $n$.

  Step 2. Let $d$ be the code of the set $\varphi\left(\U,b,c\right)$. Then $\tp(d/A)$
  is $\P$-internal and $\acl(Aa)=\acl(Ad)$.

  We have $a\in\acl(d)\subseteq\acl(Ad)$ by the definition of a code, and
  $d\in\dcl(aa_2...a_n)\subseteq \acl(Aa)$. Moreover, as $a\forkindep_A b$, we have $d\forkindep_A b$.
  Since $d$ is the code of $\varphi\left(\U,b,c\right)$
  where $\varphi$ is an $L_A$-formula, $d\in\mathrm{dcl}(Abc)$.
  Therefore $\tp(d/A)$ is $\P$-internal.
\end{proof}

\begin{proof}[Proof of Proposition \ref{equiv}]
  The nontrivial direction is from left to right. Suppose $(b_1,...,b_k)$ is an analysis of $a$ over $A$.
  For convenience, let $a_0$ be the empty tuple. We now construct the sequence $(a_1,...,a_k)$.

  Suppose  we already have $(a_1,...,a_{i-1})$ for $1\leq i\leq k$ such that $\stp(a_{j}/Aa_{j-1})$
  is $\P$-internal, $a_{j-1}\in\dcl(Aa_{j})$, and $\acl(Aa_j)=\acl(Ab_j)$ for $j=1,2,...,i-1$.
  Then as $\stp(b_{i}/Ab_{i-1})$ is almost $\P$-internal and $\acl(Aa_{i-1})=\acl(Ab_{i-1})$,
   we have that $\stp(b_{i}/Aa_{i-1})$ is almost $\P$-internal, so by Lemma \ref{lemequiv}, there exists $a^*$
  such that $\acl(Aa_{i-1}a^*)=\acl(Aa_{i-1}b_i)$ and $\tp(a^*/Aa_{i-1})$ is $\P$-internal.
  Let $a_i=a_{i-1}a^*$. Then we have
  $a_{i-1}\in\dcl(Aa_i)$, $\acl(Aa_i)=\acl(Aa_{i-1}b_i)=\acl(Ab_{i-1}b_i)=\acl(Ab_i)$,
  and $\tp(a_i/Aa_{i-1})$ is $\P$-internal.

  The sequence $(a_1,...,a_{k})$ then witnesses that $\tp(a/A)$ is strictly analysable.
\end{proof}

We use the following definitions in order to better
describe analysable types and their analyses. We say that the type $q$ is
 \emph{$k$-step $\mathcal P$-analysable}, or \emph{$\mathcal P$-analysable in $k$-steps},
if
the analysability of $q$ is witnessed by a $\mathcal P$-analysis of length $k$. A $\mathcal P$-analysis
$(a_i)_{i=1}^k$
 is said to be \emph{incompressible} if $\stp(a_{i+1}/Aa_{i-1})$ is not almost $\mathcal P$-internal
  for all $i=1,2,...,k-1$. A $\mathcal P$-analysis of $q$ is \emph{minimal} if there is no $\mathcal P$-analysis of $q$ of strictly shorter length.

  The following lemma shows that incompressibility implies minimality if the $U$-type
  of an analysis is $(1,1,...,1)$.

\begin{lemma}\label{lem1}
  Let $(a_1, ... , a_n)$ be an incompressible $\P$-analysis of $a$ over $A$ of
  $U$-type $\underbrace{(1,1,...,1)}_n$. Then the analysis is minimal, i.e.,
  $\tp(a/A)$ is not $\P$-analysable in $n-1$ steps.
\end{lemma}

\begin{proof}
  For $n=2$, the only possibility that the analysis is not minimal is that $\stp(a/A)$ is 1-step
  $\P$-analysable,
  i.e., almost $\P$-internal, which contradicts the fact that $(a_1,a_2)$ is an incompressible analysis.

  Assume we have proved the conclusion for $n<k$.
  Suppose towards a contradiction
   that $(a_1, ... , a_k)$ is an incompressible $\P$-analysis of $a$ over $A$ of $U$-type
  $\underbrace{(1,1,...,1)}_k$ which is not minimal.
  Let $(c_1, ... , c_{k-1})$ be another $\P$-analysis of $a$ over $A$. Note that $(a_1c_1,a_2c_2,...,a_{k-1}c_{k-1})$
  is also a $\P$-analysis of $a$ over $A$. Let $b_1,...,b_\ell$ be a subsequence of $(a_ic_i)_{i=1}^{k-1}$
  such that $(b_j)_{j=1}^\ell$ is a nondegenerated $\P$-analysis of $p$. This can be done by
  taking away all elements $a_ic_i$ in $(a_ic_i)_{i=1}^{k-1}$ such that $U(a_ic_i/Aa_{i-1}c_{i-1})=0$.
  Let $b_j=a$ for $\ell+1\leq j\leq k-1$. Then the only zero entries of the $U$-type of $(b_j)_{j=1}^{k-1}$
  (if any) are
   at the end
  of the sequence.

  If $U(b_1/A)=1$, then $\acl(Ab_1)=\acl(Aa_1)$, and $\stp(a/Aa_1)=\stp(a/Ab_1)$.
  But then $(a_2,...,a_k)$ is a $k-1$-step incompressible $\P$-analysis of $a$ over $Aa_1$
   of $U$-type $\underbrace{(1,1,...,1)}_{k-1}$, while $(b_2,...,b_{k-1})$ is a $k-2$-step $\P$-analysis of
   the same type with shorter length, contradicting our induction hypothesis.

   Now suppose $U(b_1/A)\geq2$. If the $U$-type of $(b_j)_{j=1}^{k-1}$
   is degenerated, then $U(b_{k-1}/b_{k-2})=0$, and we have $U(b_{k-2}/A)=U(a/A)=k$.
     If $(b_j)_{j=1}^{k-1}$
   is nondegenerated, then $U(b_j/Ab_{j-1})\geq1$ for any $j=1,...,k-2$
   which gives us $U(b_j/A)\geq j+1$ for any $j=1,...,k-2$.
    In both cases $U(b_{k-2}/A)\geq k-1$. By the induction hypothesis,
  $\acl(Ab_{k-2})\neq\acl(Aa_{k-1})$: otherwise, $(a_i)_{i=1}^{k-1}$ is
  a $k-1$-step incompressible $\P$-analysis of $a_{k-1}$ over $A$
  of $U$-type $\underbrace{(1,1,...,1)}_{k-1}$, while $(b_i)_{i=1}^{k-2}$ is a $k-2$-step $\P$-analysis
  of the same type,
   contradicting our induction hypothesis. Similarly, $\acl(Ab_{k-2})
  \supsetneq\acl(Aa_{k-1})$ does not hold: otherwise $U(b_{k-2}/Aa_{k-1})\geq1$, and since $b_{k-2}\in\acl(Aa)$ and $U(a/Aa_{k-1})=1$, we have $\acl(Ab_{k-2})=\acl(Aa)$;
  therefore $(a_2,...,a_{k})$ is a $k-1$-step incompressible $\P$-analysis of $\stp(a/Aa_1)$
   of $U$-type $\underbrace{(1,1,...,1)}_{k-1}$, while $(b_1,...,b_{k-2})$ is a $k-2$-step $\P$-analysis of
   the same type, contradicting our induction hypothesis.
  Hence $\acl(Ab_{k-2})
  \supseteq\acl(Aa_{k-1})$ does not hold, i.e., $a_{k-1}\not\in\acl(Ab_{k-2})$. We have
  $k=U(a/A)\geq U(a_{k-1}b_{k-2}/A)=U(b_{k-2}/A)+U(a_{k-1}b_{k-2}/Ab_{k-2})\geq (k-1)+1=k$, so
  $\acl(Ab_{k-2}a_{k-1})=\acl(Aa)$. But then since
  $\stp(b_{k-2}/Aa_1)$ and $\stp(a_{k-1}/Aa_1)$ are $k-2$-step $\P$-analysable, so is
  $\stp(a/Aa_1)$, while $(a_2,...,a_k)$ is a $k-1$-step incompressible $\P$-analysis of $a$ over $Aa_1$
  of $U$-type $\underbrace{(1,1,...,1)}_{k-1}$,
   contradicting our induction hypothesis.
\end{proof}

\section{Iterated Logarithmic Derivative}\label{siter}
\noindent
Our primary interest is in $\DCF$, the theory of differential closed field of
 characteristic 0.
 The theory $\DCF$ is complete, stable, and eliminates both quantifiers and imaginaries.
 We assume some familiarity of this theory.
 The language used is $(0,1,+,\times,\delta)$, and
 $\U=(U,0,1,+,\times,\delta)$ is the saturated
 model, where $\delta$ is the derivative on the field. We often omit $0,1,+,\times$ and write $\U=(U,\delta)$.

 We focus on types which are almost $\C$-internal or $\C$-analysable in $\DCF$, where
 $\C=\{x:\delta x=0\}$ is the field of constants.

We often use the term ``generic type'' in $\DCF$. The \emph{generic type} of an irreducible Kolchin closed set $D$
over a $\delta$-field $k$ is the type which
says that $x$ is in $D$ but not in any $k$-definable Kolchin closed subset of $D$.
A definable set is \emph{irreducible} if its Kolchin closure is. By the generic type of an irreducible
definable set,
we mean the generic type of its Kolchin closure. Note that this does not always coincide with the type of greatest $U$-rank.

Recall that in $\DCF$, the logarithmic derivative of $x$ is defined as $\displaystyle\ld x=\frac {\delta x}x$.
The logarithmic derivative is used extensively in this section. Note that $\ld:\mathbb G_m\rightarrow\mathbb G_a$
is a definable group homomorphism between algebraic groups, and the kernel of the map
is $\mathbb G_m(\C)$. Here $\mathbb G_m$ is the universe (take away 0) viewed as a multiplicative group,
$\mathbb G_a$ is
the universe viewed as an additive group, and $\mathbb G_m(\C)$ is the constant points of $\mathbb G_m$.

\begin{fact}[see, for example, Fact 4.2 of \cite{chatzidakis2015differential}]\label{fact1}
  Let $G$ be the differential algebraic subgroup of $\mathbb G_m$ defined by $\{x:\delta(\ld x )=0\}$.
   The generic type of $G$ is 2-step $\C$-analysable but not almost $\C$-internal.
\end{fact}

It follows that any $\C$-analysis of this type is of $U$-type $(1,1)$.

We will be considering iterated logarithmic derivatives. For any $n\geq 1$ we set $\ld^{(n)}(x):=\underbrace{\ld~\ld~...~\ld(x)}_{n\text{~times}}$.
Note that $\ld^{(n)}(x)$ is only defined at $x$ if $\ld^{(i)}(x)\neq0$ for $i=0,1,...,n-1$ where $\ld^{(0)}(x)=x$.
Whenever we write $\ld^{(n)}(x)$ it is always assumed that $x$ is in this domain of definition. Note that for any $h\in\U$, the equation $\ld^{(n)}(x)=h$ defines an irreducible Kolchin constructible subset $B$ of $\U$. Indeed, $B$
 is isomorphic to
 \begin{equation*}\begin{split}
 B^*=&\{(x,\ld(x),...,\ld^{(n-1)}(x)):x\in B\}\\
 =&\{(x_1,...,x_n):x_i\neq0;\frac{\delta x_i}{x_i}=x_{i+1},i=1,2,...,n-1;\frac{\delta x_n}{x_n}=h\}
 \end{split}\end{equation*}
  whose Kolchin closure is
 $\{(x_1,...,x_n):\delta x_i=x_ix_{i+1},i=1,2,...,n-1;\delta x_n=hx_n\}$, which is irreducible since
 it is the set of $D$-points (or ``sharp'' set) corresponding to the irreducible $D$-variety
 $(\mathbb A^n,s)$ where $s(x_1,...,x_{n-1},x_n)=(x_1x_2,...,x_{n-1}x_n,hx_n)$.
 (For details on $D$-varieties see \cite{kowalskipillay}.)

 In particular, $\{x:{\ld}^{(2)}(x) =h\}$ is irreducible.
  Note also that the generic type of $\ld^{(2)}(x)=0$ is the same as that of $G$ defined in Fact \ref{fact1}.
So the following proposition is a generalisation of Fact \ref{fact1}.

\begin{proposition}\label{fprop1}
  Let $h\in U$ and consider $B=\{x:{\ld}^{(2)}(x) =h\}$.
  Let $k$ be a $\delta$-field containing $h$, and $p$ be the generic type of $B$ over $k$.
  Then $p$ is not almost $\C$-internal.
\end{proposition}

\begin{proof}
  We may assume that $k$ contains an element of the form $a=\ld g_0$  where $g_0\in B$. Indeed, this follows
  from the fact that for any $g_0\in B$, $p$ is almost $\C$-internal iff the non-forking extension of $p$ to
   $k\langle g_0\rangle$ is, and $p|k\langle g_0\rangle$  is the generic type of $B$ over $k\langle g_0\rangle$.

  We now construct a new model $\mathcal V=(U,D)$ of $\DCF$ as follows.
  The set $U$ and the interpretation of $0,1,+$ and $\times$ remain the same, while
  $\displaystyle Dg:=\frac{\delta g}{a}$ for all $g\in\U$.
  Notice that $\mathcal V$ is also a model of $\DCF$ with the same field of constants as $\U$, and any definable set
  in one model is definable in the other, with the same set of parameters, as long as the parameter set contains $a$.
  Now let $q$ be a type in the model $\mathcal V$ over $k$ so that $q$ and $p$ have the same set of realizations
  in $U$.
  This can be done by replacing each occurrence of $\delta$ in formulas in $p$ by $aD$.

  Assume towards a contradiction that $p$ is almost $\C$-internal. Hence, for any $g\models p$, there is $B\supset k$
  such that $g\forkindep_{k} B$ and $g\in \acl(BC)$, in the model $\U$.
  Replacing $\delta$ by $aD$ in the formulas witnessing this fact, we have that $g\in\acl(BC)$ in $\mathcal V$ as well.
  Moreover, $g\forkindep_{k} B$ holds in $\mathcal V$ because $U$-ranks of types are the same in $\U$ and
  $\mathcal V$ if the parameter set contains $a$.
  We get that $q$ is almost $\C$-internal in $\mathcal V$.

  However, $q$ is the generic type of $B$, since Kolchin closed sets definable over $k$
   (which contains $a$) are the same in $\U$ and $\mathcal V$. The set $B$ is defined in $\U$ by the formula
   ${\ld}({\ld}x) =h$, which is just ${a\log D}({\log D}x) =h$, which is equivalent to $\log D({\log D}f) =0$.
So $q$ is the generic type of $B=\{x:\log D({\log D}x) =0\}$, which is not almost $\C$-internal in $\mathcal V$
by Fact \ref{fact1}, a contradiction.
\end{proof}

We can now show that the iterated logarithmic derivatives give rise to $n$-step $\C$-analysable types that are not
$n-1$-step $\C$-analysable.

\begin{corollary}\label{iter}
  In $\DCF$, let $D=\{x\in U : \underbrace{\ld~ \ld ~... ~\ld}_n x=0\}$.
  Then the generic type $p$ of $D$ is $n$-step $\C$-analysable
  but not $n-1$-step $\C$-analysable.
\end{corollary}

\begin{proof}
  Let $a\in D$ be generic. Let $a_n=a$, $a_k=\ld a_{k+1}$ for $k=0,1,...,n-1$.
  Note that $a_0=0$, $a_k\in\dcl(a_{k+1})$ for $k=0,1,...,n-1$, and $a$ is interdefinable with $(a_1,...,a_n)$.

  As $a$ is generic in $D$, $a_{i+1}\not\in\acl(a_i)$ for each $i=0,1,...,n-1$.
  By additivity of $U$-rank, for each $i=0,1,...,n-1$, $U(a_{i+1}/a_i)=1$. Hence, $\stp (a_{i+1}/a_i)$
  is the generic type over $a_i$ of $\ld(x)=a_i$.  The latter equation defines a multiplicative translation of
  $\mathbb G_m(\C)=\ker(\ld)$, so $\stp(a_{i+1}/a_i)$ is almost $\C$-internal of $U$-rank 1.
  That is, $(a_1,a_2,...,a_n=a)$ is a $\C$-analysis of $p$ of $U$-type $\underbrace{(1,1,...,1)}_n$.

  For each $i=1,2,...,n-1$, $\stp(a_{i+1}/a_{i-1})$ is the generic type of $\ld^{(2)}x=a_{i-1}$ over $a_{i-1}$.
  Proposition \ref{fprop1} tells us that this type is not almost $\C$-internal. That is, $(a_1,a_2,...,a_n)$
  is an incompressible $\C$-analysis.

  Hence, by Lemma \ref{lem1}, $p$ is not $\C$-analysable in $n-1$ steps.
\end{proof}

\section{Analyses by reductions and coreductions}\label{redu}
\noindent
In this section we return to the general setting of Section \ref{sprel}; so $T$ is a complete stable theory that
eliminates imaginaries, $\U$ is a sufficiently saturated model of $T$, and $\P$ is a set of partial types invariant
over automorphisms of the universe.

Note that Lemma \ref{lem1} does not hold if the entries of the $U$-type are not
 all 1.

\begin{example}\label{s4example}
  Let $\stp(a)$ be
  2-step $\P$-analysable with an incompressible $\P$-analysis $(a_1,a)$.
  Now let $(b_1,b)$ be such that $bb_1\forkindep aa_1$ and $\stp(bb_1)=\stp(aa_1)$.
  Let $c=ab$. Then $c$ is 3-step $\P$-analysable, with an analysis $(a_1, ab_1, c=ab)$.
  This analysis is incompressible: $\stp(ab_1)$ is not almost $\P$-internal because $\stp(a)$ is not almost
  $\P$-internal
  and $\stp(ab/a_1)$ is not almost $\P$-internal because $\stp(b)$ is not almost $\P$-internal, and $\stp(b/a_1)$ is its
  non-forking extension.
   But $c$ is 2-step $\P$-analysable by $(a_1b_1, c=ab)$, so the $\P$-analysis $(a_1, ab_1, c=ab)$ is
   not minimal despite being incompressible.
\end{example}

  To generalize Lemma \ref{lem1} to higher $U$-rank cases, we need each step to satisfy some maximality or
  minimality property. We will use the notions of $\P$-reduction and $\P$-coreduction.

\begin{definition}[See, for example, Section 4 of \cite{moosa2010model}]
  Suppose $a$ is a tuple and $A$ is a parameter set.
  We say a tuple $b$ is a \emph{$\P$-reduction of $a$ over $A$} if $b$ is maximally
  almost $\P$-internal over $A$ in $\acl(Aa)$,
  i.e., $\stp(b/A)$ is almost $\P$-internal,
  $b\in\acl(Aa)$, and if $b'\in \acl(Aa)$ and $\stp(b'/A)$ is almost $\P$-internal
  then $b'\in\acl(Ab)$. We say a nondegenerated $\P$-analysis $(a_1,...,a_n)$
  of $a$ over $A$ is a \emph{$\P$-analysis by reductions of $a$ over $A$} if
  $a_k$ is the $\P$-reduction of $a$ over $Aa_{k-1}$ for $k=1,2,...,n$.
\end{definition}

Note that by definition $\P$-reductions are unique up to interalgebraicity over the parameter set, i.e.,
if $b$ and $c$ are both $\P$-reductions of $a$ over $A$, then $\acl(Ab)=\acl(Ac)$.
We may therefore call $b$ \emph{the} $\P$-reduction of $a$ over $A$.

\begin{remark}
It is clear that if $U(a/A)<\omega$, then a $\P$-reduction of $a$ over $A$ always exists. In fact,
  let $b$ be a tuple that has maximal $U$-rank over $A$ satisfying
  the condition that $\stp(b/A)$ is almost $\P$-internal and $b\in\acl(Aa)$. Then $b$ is a $\P$-reduction
  of $a$ over $A$: if $c$ also satisfies this condition, then $\stp(bc/A)$ is
  almost $\P$-internal and $bc\in\acl(Aa)$,
  so $U(bc/A)=U(b/A)$, which means $c\in\acl(Ab)$. Hence, if $\tp(a/A)$ is
  $\P$-analysable of finite $U$-rank then a $\P$-analysis by reductions always exists.
\end{remark}

\begin{definition}[See, for example, Definition  4.1 of \cite{moosa2010model}]
  Suppose $a$ is a tuple and $A$ is a parameter set.
  We say a tuple $b$ is a \emph{$\P$-coreduction of $a$ over $A$} if $b$ is minimal in $\acl(Aa)$ such that $a$ is
  almost $\P$-internal over $Ab$,
  i.e., $\stp(a/Ab)$ is almost $\P$-internal,
  $b\in\acl(Aa)$, and if $b'\in \acl(aA)$ and
  $b'$ satisfies that $\stp(a/Ab')$ is almost $\P$-internal
  then $b\in\acl(Ab')$. We say a nondegenerated $\P$-analysis $(a_1,...,a_n)$ of $a$ over $A$
  is a \emph{$\P$-analysis by coreductions of $a$ over $A$} if
  $a_{k-1}$ is a $\P$-coreduction of $a_k$ over $A$ for $k=2,...,n$.
\end{definition}

Note similarly that by definition $\P$-coreductions are unique up to interalgebraicity over the parameter set.
We may therefore call $b$ \emph{the} $\P$-coreduction of $a$ over $A$.

Recall that $T$ has the \emph{canonical base property} (CBP) if whenever $U(a/b)<\omega$ and
  $\acl(b)=\acl(\mathrm{Cb}(a/b))$, then $\stp(b/a)$ is almost $\mathbb P$-internal, where $\mathbb P$
  is the set of all nonmodular minimal types.
  See, for example, Section 1 of \cite{moosa2008canonical}.
  It is a fact that if $T$ has CBP then $\mathbb P$-coreductions
  exist for any finite-rank type (see Theorem 2.4 of \cite{chatzidakis2012note}). Hence, assuming $T$ has CBP,
  if $\stp(a/A)$ is $\mathbb P$-analysable of finite $U$-rank
  then a $\mathbb P$-analysis by coreductions always exists.

The following lemma shows that in $\DCF$£¬ $\C$-coreductions of any finite-rank type always exist.
This is because any nonmodular minimal type in $\DCF$ is almost $\C$-internal.

\begin{lemma}\label{dcf}
  We work in $\DCF$ in this lemma. If  $U(a/A)$ is finite, then the $\C$-coreduction of $a$ over $A$ exists.
\end{lemma}

\begin{proof}
  Let $\mathbb P$ be the set of all nonmodular minimal types in $\U\models \DCF$.
  By Theorem 1.1 of \cite{pillay2003jet}, $\DCF$ has CBP. Therefore, there exists $b$ which is
  the $\mathbb P$-coreduction of $a$ over $A$.

  We want to show that $b$  is the $\C$-coreduction of $a$ over $A$. In fact, we only need to show that
  if a type is almost $\mathbb P$-internal then it is almost $\C$-internal. Suppose $\tp(e/D)$ is $\mathbb P$-internal.
  Then for some $B\supset D$ such that $B\underset{D}\forkindep e$ and a tuple $c$ consists of realizations of types in $\mathbb P$ with bases in $B$,
  $e\in\acl(Bc)$. Since every minimal nonmodular type in $\DCF$ is almost $\C$-internal,
  there exist $F\supset B$ such that $F\underset{B}\forkindep ec$ and $c\in\acl(F\C)$. Now
  $e\in\acl(Bc)\subseteq\acl(F\C)$,
  and since $e\underset{B}\forkindep F$ and $e\underset{D}\forkindep B$, we have $e\underset{D}\forkindep F$.
  This shows that $\tp(e/D)$ is almost $\C$-internal.
\end{proof}

It is not hard to see that analyses by reductions or coreductions are incompressible.
If $(a_1,...,a_n)$ is a $\P$-analysis by reductions of $\tp(a/A)$
and $\stp(a_{i+1}/Aa_{i-1})$ is almost $\P$-internal for some $i=1,2,...,n-1$, then since $a_i$ is the $\P$-reduction of $a$ over $Aa_{i-1}$, $a_{i+1}\in\acl(Aa_i)$ which
implies $\acl(Aa_i)=\acl(Aa_{i+1})$.
Now for any $j>i$, assume that $\acl(Aa_j)=\acl(Aa_i)$.
Then since $a_{j+1}$ is the $\P$-reduction of $a$ over $Aa_j$ and $\acl(Aa_j)=\acl(Aa_i)$,
$a_{j+1}$ is the $\P$-reduction of $a$ over $Aa_i$, so $\acl(Aa_{j+1})=\acl(Aa_{i+1})=\acl(Aa_i)$.
Thus $a_i,...,a_n$ are all the same up to
 interalgebraicity over $A$, and this is possible only if $i=n$, contradicting the fact that $i\leq n-1$.
Similarly, if $(a_1,...,a_n)$ is a $\P$-analysis by coreductions of $\tp(a/A)$
and $\stp(a_{i+1}/Aa_{i-1})$ is almost $\P$-internal for some $i=1,2,...,n-1$, then since $a_i$
is the $\P$-coreduction of $a_{i+1}$ over $a_{i-1}$, $a_{i}\in\acl(Aa_{i-1})$ which
implies $a_i$ and $a_{i-1}$ are interalgebraic over $A$. An
inductive argument similar to the reduction case shows that $a_0,...,a_i$ are all the same up to
 interalgebraicity over $A$, and this is possible only if $i=0$, contradicting the fact that $i\geq 1$.

However, more is true: they are actually minimal.

\begin{proposition}\label{minimal}
  Analysis by reductions and coreductions are minimal.
\end{proposition}

\begin{proof}
  Let $(a_1,...,a_n)$ and $(c_1,...,c_\ell)$ be $\P$-analyses of $a$ over $A$
  with $(a_1,...,a_n)$ being by reductions. We shall prove that $n\leq \ell$.
  We show that $c_i\in\acl(Aa_i)$ for $i=1,2,...,\min(n,\ell)$. For $i=1$, since $\stp(c_1/A)$ is
  almost $\P$-internal and $a_1$ is the $\P$-reduction of $a$ over $A$, $c_1\in\acl(Aa_1)$. Now if $c_{i-1}\in\acl(Aa_{i-1})$, then $\stp(c_i/a_{i-1})$ is almost $\P$-internal, and as
  $a_i$ is the $\P$-reduction of $a$ over $Aa_{i-1}$, $c_i\in\acl(Aa_i)$ as desired.
  Suppose $\ell<n$. Then $\acl(Aa_{\ell})\subsetneq \acl(Aa_n) $ since $(a_1,...,a_n)$ is incompressible, so
   $\acl(Aa)=\acl(Ac_{\ell})\subseteq \acl(Aa_{\ell})\subsetneq \acl(Aa_n)=\acl(Aa)$,
   a contradiction.

  Now suppose $(b_1,...,b_m)$ is a $\P$-analysis by coreductions of $a$ over $A$.
  We shall prove that $m\leq\ell$.
  We show that $b_{m-j}\in\acl(Ac_{\ell-j})$ for $j=0,1,...,\min(m,\ell)-1$.
  For $j=0$, notice that $b_m,c_\ell$ are both interalgebraic over $A$ with $a$. Now if
  $b_{m-j+1}\in\acl(Ac_{\ell-j+1})$, then $\stp(b_{m-j+1}/c_{\ell-j})$ is almost $\P$-internal, and as
  $b_{m-j}$ is the $\P$-coreduction of $b_{m-j+1}$ over $A$, $b_{m-j}\in\acl(Ac_{\ell-j})$ as desired.
  Assume towards a contradiction that  $\ell<m$. Then
   $\acl(Ab_{m-\ell+1})\subseteq \acl(Ac_1)$. Since $m-\ell+1\geq 2$, $\stp(b_{m-\ell+1}/A)$
   is not almost $\P$-internal because $(b_1,...,b_m)$ is incompressible, but $\stp(c_1/A)$ is almost $\P$-internal,
   a contradiction.
\end{proof}

So analyses by reductions and coreductions are of the same length.
However, analyses by reductions and coreductions do not always have to agree (even up to interalgebraicity).

\begin{definition}
  We say that two $\P$-analyses $(a_1,...,a_n)$ and $(b_1,...,b_m)$ of $a$ over $A$ are \emph{interalgebraic
  over $A$} if $n=m$ and $\acl(Aa_i)=\acl(Ab_i)$ for $i=1,2,...,n$.
  We call an analysis \emph{canonical} if it is minimal and interalgebraic with every other minimal analysis.
\end{definition}

\begin{example}
Using the notation of Example \ref{s4example}, the $\P$-analysis by reductions of $ab_1$
over $\varnothing$ is $(a_1b_1,ab_1)$, while the $\P$-analysis by coreductions of $ab_1$ is $(a_1,ab_1)$.
But $(a_1b_1,ab_1)$ and $(a_1,ab_1)$ are not interalgebraic. In particular, $\stp(ab_1)$ does
not have a canonical $\P$-analysis.
\end{example}

The next proposition points out, however, that if an analysis by reductions has the same $U$-type as one by coreductions, then they are interalgebraic and are in fact the unique minimal analysis up to interalgebraicity.

\begin{proposition}\label{unique}
  Let $(a_1,...,a_n)$ and $(b_1,...,b_n)$ be $\P$-analyses by reductions and coreductions of $a$ over $A$,
  respectively. If
  the $U$-types of $(a_1,...,a_n)$ and $(b_1,...,b_n)$
  are the same, then $(a_1,...,a_n)$ is interalgebraic with $(b_1,...,b_n)$ over $A$. Moreover,
   if $(c_1,...,c_n)$
  is another $\P$-analysis of $a$ over $A$, then $(c_1,...,c_n)$ is also interalgebraic with both
  $(a_1,...,a_n)$ and $(b_1,...,b_n)$ over $A$.

  In particular, if $p$ has an analysis by reductions and an analysis by coreductions of the same
  $U$-type, then these analyses are canonical. Conversely, any canonical analysis is an analysis
  by both reductions and coreductions.
\end{proposition}

\begin{proof}
  Having the same $U$-type implies that $U(a_i/A)=U(b_i/A)$ for $i=1,2,...,n$.
  Let $(c_1,...,c_n)$ be another $\P$-analysis of $a$ over $A$,
   We have seen in the proof of \ref{minimal} that $c_i\in\acl(Aa_i)$ and $b_{i}\in \acl(Ac_i)$ for $i=1,2,...,n$.
  Therefore $U(a_i/A)=U(b_i/A)=U(c_i/A)$ and  $\acl(Aa_i)=\acl(Ab_i)=\acl(Ac_i)$ for $i=1,2,...,n$, as desired.

  The ``in particular'' clause now follows by Proposition \ref{minimal}. For the converse,
   let $(a_i)_{i=1}^n,(b_i)_{i=1}^n,(c_i)_{i=1}^n$ be
$\P$-analyses of $a$ over $A$, which are an analysis by reductions, an analysis by coreductions, and a canonical
analysis, respectively. We have that $a_i$ is the $\P$-reduction of $a$ over $Aa_{i-1}$, $\acl(Aa_i)=\acl(Ac_i)$,
and $\acl(Aa_{i-1})=\acl(Ac_{i-1})$, so $c_i$ is the $\P$-reduction of $a$ over $Ac_{i-1}$. Thus $(c_i)_{i=1}^n$
is a $\P$-analysis by reductions. Similarly, we have that $b_i$ is the $\P$-coreduction of $b_{i+1}$ over $A$, $\acl(Ab_i)=\acl(Ac_i)$,
and $\acl(Ab_{i+1})=\acl(Ac_{i+1})$, so $c_i$ is the $\P$-coreduction of $a$ over $Ac_{i-1}$. Thus $(c_i)_{i=1}^n$
is a $\P$-analysis by coreductions.
\end{proof}

Here is a local criterion to determine whether an analysis is an analysis by reductions.

\begin{lemma}\label{lemcr1}
  Let $(a_1,...,a_n)$ be a $\P$-analysis of $a$ over $A$.
  Then it is a
   $\P$-analysis by reductions iff $a_i$ is a $\P$-reduction of $a_{i+1}$ over $Aa_{i-1}$
   for $i=1,...,{n-1}$.
\end{lemma}

\begin{proof}
  Suppose $(a_1,...,a_n)$ is a $\P$-analysis by reductions of $a$ over $A$. For any $k=1,2,...,n-1$,
  $a_k$ is a $\P$-reduction of $a$ over $Aa_{k-1}$, i.e., for any $a_k'\in\acl(Aa)$, if
  $\stp(a_k'/Aa_{k-1})$ is almost $\P$-internal, then $a_k'\in\acl(a_k)$. In particular,
  for any $a_k'\in\acl(Aa_{k+1})$, if
  $\stp(a_k'/Aa_{k-1})$ is almost $\P$-internal, then $a_k'\in\acl(a_k)$. Note that $a_k\in\acl(Aa_{k+1})$,
  so $a_k$ is a $\P$-reduction of $a_{k+1}$ over $Aa_{k-1}$.

  Now suppose $(a_1,...,a_n)$ is a $\P$-analysis of $a$ over $A$ such that
  $a_i$ is a $\P$-reduction of $a_{i+1}$ over $Aa_{i-1}$
   for $i=1,...,{n-1}$.
  We need to check that $a_k$ is the $\P$-reduction of $a$ over $Aa_{k-1}$. In fact, let $a_k'$ be
  the $\P$-reduction of $a$ over $Aa_{k-1}$, then we only need to show that $a_k'\in \mathrm{acl}(Aa_k)$.

  We know $a_k'\in\mathrm{acl}(Aa_n)$. Suppose $a_k'\in\mathrm{acl}(Aa_i)$ for some $i$ such
  that $k<i\leq n$.
  Since $a_k'$ is almost $\P$-internal over $Aa_{k-1}$ and $k-1<i-1$,
  $a_k'$ is $\P$-internal over $Aa_{i-2}$. Now $a_{i-1}$ is a $\P$ reduction of $a_i$ over $Aa_{i-2}$,
  $a_k'\in\mathrm{acl}(Aa_i)$, and $a_k'$ is almost $\P$-internal over $Aa_{i-2}$,
    so $a_k'\in\mathrm{acl}(Aa_{i-1})$.
  By induction we get $a_k'\in\mathrm{acl}(Aa_{k})$.
\end{proof}

We have a similar criterion for analyses by coreductions.

\begin{lemma}\label{lemcr2}
   A $\P$-analysis $(a_1,...,a_n)$ of $a$ over $A$ is a
   $\P$-analysis by coreductions iff $a_i$ is a $\P$-coreduction of $a_{i+1}$ over $Aa_{i-1}$
   for $i=1,...,{n-1}$.
\end{lemma}

\begin{proof}
  Suppose $(a_1,...,a_n)$ is a $\P$-analysis by coreductions of $a$ over $A$. For any $k=1,2,...,n-1$,
  $a_k$ is a $\P$-coreduction of $a_{k+1}$ over $A$, i.e., for any $a_k'\in\acl(Aa_{k+1})$, if
  $\stp(a_{k+1}/Aa_k')$ is $\P$-internal, then $a_k\in\acl(Aa_k')$. In particular,
  for any $a_k'\in\acl(Aa_{k+1})$, if
  $\stp(a_{k+1}/Aa_{k-1}a_k')$ is $\P$-internal, then $a_k\in\acl(Aa_{k-1}a_k')$.
  So we have that $a_k$ is a reduction of $a_{k+1}$ over $Aa_{k-1}$.

  Now suppose $(a_1,...,a_n)$ is a $\P$-analysis of $a$ over $A$ such that
  $a_i$ is a $\P$-coreduction of $a_{i+1}$ over $Aa_{i-1}$
   for $i=1,...,{n-1}$.
  Fixing a $k\in\{1,2,...,n-1\}$,
  we need to check that $a_k$ is the $\P$-coreduction of $a_{k+1}$ over $A$. In fact, let $a'$ be
  be such that $\stp(a_{k+1}/Aa')$ is almost $\P$-internal. We need to prove that $a_k\in \mathrm{acl}(Aa')$.

  We know that $a_1\in \mathrm{acl}(Aa')$. This is because $a_1$ is the $\P$-coreduction of $a_2$ over $A$,
  and $\stp(a_2/Aa')$ is almost $\P$-internal (since $a_2\in\dcl(Aa_{k+1}))$.

  Suppose $a_{i-1}\in\mathrm{acl}(Aa')$ for some $i$
  such that $1< i\leq k$.
  Since $a_{i+1}$ is almost $\P$-internal over $Aa'$ (as $i+1\leq k+1$,
  $a_{i+1}\in\mathrm{acl}(Aa_{k+1})$), and $a_i$ is the $\P$-coreduction of $a_{i+1}$ over $Aa_{i-1}$, we have that
  $a_i\in \mathrm{acl}(Aa')$. By induction we get $a_k\in \mathrm{acl}(Aa')$.
\end{proof}

It follows from the above lemma that an incompressible analysis of $U$-type $(1,1,...,1)$ is canonical.
Indeed, for such an analysis $(a_1,...,a_n)$ of $a$ over $A$, as $\stp(a_{i+1}/Aa_{i-1})$ is not almost
$\P$-internal, by rank consideration, $a_i$ must be both the $\P$-reduction and the $\P$-coreduction of $a_{i+1}$
over $Aa_{i-1}$ for $i=1,2,...,n-1$.

We end this section by pointing out that once we have a type with an incompressible analysis of $U$-type
$\underbrace{(1,1,...,1)}_n$ --  as for example we do in $\DCF$ by Corollary \ref{iter} -- then every
decreasing sequence of positive integers of length $n$ appears as the $\U$-type of the $\P$-analysis by reductions of
some other type in this theory. A similar statement holds for increasing sequences and $\P$-analyses by coreductions provided that every finite $U$-rank type has a $\P$-coreduction.
For convenience we work over the empty set.

\begin{proposition}\label{seqred}
  Suppose $(a_1,...,a_n)$ is a $\P$-analysis of $a$ of $U$-type $(1,1,...,1)$.
  \begin{enumerate}[label=(\alph*)]
    \item Given positive integers $s_1\geq...\geq s_n$, there exists a tuple whose $\P$-analysis by reductions is of
    $U$-type $(s_1,...,s_n)$.
    \item Suppose every type of finite $U$-rank has a $\P$-coreduction. Given positive integers $s_1\leq...\leq s_n$, there exists a tuple whose $\P$-analysis by coreductions is of $U$-type $(s_1,...,s_n)$.
  \end{enumerate}
\end{proposition}

\begin{proof}
  (a) Let $\bar a^{(j)}=(a_1^{(j)},...,a_n^{(j)})$, $j=1,2,...$ be tuples
   such that $(\bar a^{(1)},\bar a^{(2)},...)$ is a Morley sequence of $\tp(a_1,...,a_n)$.
   In particular, $a_i^{(j)}$ is the $\P$-reduction and the $\P$-coreduction of $a_{i+1}^{(j)}$.
  Let $\alpha_i=(a_i^{(1)},...,a_i^{(s_i)})$ and $\beta_i=(\alpha_1,...,\alpha_i)$.
  Note that $a_i^{(j)}\in\beta_i$ for $j=1,2,...,s_i$.
  We claim the tuple
  $\beta_n$ is $\P$-analysable and its $\P$-analysis by reductions is of $U$-type $(s_1,...,s_n)$.
  To show this, since $(\bar a^{(j)})_j$ is a Morley sequence, we have
  \begin{equation*}\begin{split}
  U(\beta_i/\beta_{i-1})=&U(\alpha_i/\beta_{i-1})\\
  =&U(a_i^{(1)}...a_i^{(s_i)}/\beta_{i-1})\\
  =&U(a_i^{(1)}...a_i^{(s_i)}/a_{i-1}^{(1)}a_{i-1}^{(s_i)})\\
  =&s_i,
  \end{split}\end{equation*}
  so we only need to prove that the $\P$-analysis by reductions of $\beta$ is
  $(\beta_1,\beta_2,...,\beta_n)$.

  Let $b_i$ be the reduction of $\beta_n$ over $\beta_{i-1}$. We claim that $b_i$ is
  interalgebraic with $\beta_{i}$.
  Since $a_{i-1}^{(j)}\in\dcl(\beta_{i-1})$ for $j=1,2,...,s_i$ (since $s_{i-1}\geq s_i$),
  $\stp(a_{i}^{(j)}/\beta_{i-1})$ is almost $\P$-internal for $j=1,2,...,s_i$, so
  $\stp(\alpha_i/\beta_{i-1})$ is almost $\P$-internal.
  Since
  $\beta_i\in\dcl(\alpha_i,\beta_{i-1})$,
  $\stp(\beta_i/\beta_{i-1})$
  is almost $\P$-internal, so $\beta_i\in b_i$. We now need to show that
  $U(b_i/\beta_{i})=0$. Toward a contradiction, suppose $U(b_i/\beta_{i})>0$.

  Set $B=\beta_i$, which is the collection of elements of the form $a_p^{(q)}$ where $1\leq p\leq i$ and $1\leq q\leq s_i$.
  Now we add elements of the form $a_p^{(q)}$ one by one into $B$ according to dictionary order of $(p,q)$ where
  $i+1\leq p\leq n$ and $1\leq q\leq s_i$ as long as $U(b_i/B)$ remains unchanged. Since $b_i\in \beta_n$,
  $U(b_i/\beta_n)=0$, so this process will terminate for some $a_p^{(q)}$ where $U(b_i/Ba_p^{(q)})<U(b_i/B)$.
  
  Now $B$ contains elements of the form $a_{p'}^{(q')}$ where $(p',q')<(p,q)$ by dictionary order. We have $a_p^{(q)}\underset{B}{\not\forkindep} b_i$.
   As $a_{p-1}^{(q)}\in B$ and $a_{p}^{(q)}\underset{a_{p-1}^{(q)}}\forkindep B$, $U(a_p^{(q)}/B)=1$,
   so $a_p^{(q)}\in \acl (Bb_i)$. However, Let $C=\{a_i^{(j)}:a_{i+1}^{(j)}\in \dcl(B)\}$. Then $\stp(B/C)$ is almost $\P$-internal as $\stp(a_{i+1}^{(j)}/a_{i}^{(j)})$ is almost internal for any $i,j$,
  and $\stp(b_i/C)$ is almost $\P$-internal because $\beta_{i-1}\in\dcl(C)$. But $\stp(a_p^{(q)}/C)$ is not almost
  $\P$-internal: since $a_{p-1}^{(q)}\not\in \acl(a_{p-2}^{(q)})$ and $\displaystyle{ a_{p-1}^{(q)}\underset{{a_{p-2}^{(q)}}}\forkindep C}$,
  we have $a_{p-1}^{(q)}\not\in\acl(C)$.

  (b)
  Let $\bar a^{(j)}=(a_1^{(j)},...,a_n^{(j)})$, $j=1,2,...$ be tuples
   such that $(\bar a^{(1)},\bar a^{(2)},...)$ is a Morley sequence of $\tp(a_1,...,a_n/A)$.
  Let $\beta_i=(a_1^{(1)}...a_1^{(s_{n-i+1})},...,a_i^{(1)}...a_i^{(s_1)})$.
  Let $f(j)=\min\{k:j\leq s_k\}$, and let $f(j)$ be infinity if it is not defined.
  Then $a_{k}^{(j)}\in\beta_i$ iff $k\leq i-f(j)+1$ and $\displaystyle\beta_{i}=\bigcup_{j=1}^{s_{i}}a_{i+1-f(j)}^{(j)}$.
  We claim the tuple
  $\beta_n$ is $\P$-analysable and its $\P$-analysis by coreductions is of $U$-type $(s_1,...,s_n)$.
  Since
  $\displaystyle\beta_{i}=\bigcup_{j=1}^{s_{i}}a_{i+1-f(j)}^{(j)}$ and
  $\displaystyle\beta_{i-1}=\bigcup_{j=1}^{s_{i}}a_{i-f(j)}^{(j)}$
  (as $i-f(j)=0$ for $s_{i-1}<j\leq s_i$, we may set the upper bound as $s_i$), we have
  \begin{equation*}\begin{split}
    U(\beta_i/\beta_{i-1})=&U(\bigcup_{j=1}^{s_{i}}a_{i+1-f(j)}^{(j)}/\beta_{i-1})\\
    =&\sum_{j=1}^{s_i}U(a_{i+1-f(j)}^{(j)}/a_{i-f(j)}^{(j)})\\
     =&s_i
  \end{split}\end{equation*}
  as $(\bar a^{(j)})_j$ is a Morley sequence.
  Thus we only need to prove that the $\P$-analysis by coreductions of $\beta$ is
  $(\beta_1,\beta_2,...,\beta_n)$.

  Suppose $b$ is the $\P$-coreduction of $\beta_{i+1}$ over the empty set.
  We claim that $\acl(b)=\acl(\beta_i)$.
   Note that $\stp(\beta_{i+1}/\beta_i)$
  is almost $\P$-internal, so $b\in\acl(\beta_i)$.
  Take any $a_j^{(k)}\in \beta_i$. Since $a_{j+1}^{(k)}\in\beta_{i+1}$
  and $\beta_{i+1}$ is almost $\P$-internal over $b$,
  $a_{j+1}^{(k)}$ is almost $\P$-internal over $b$, so $a_{j}^{(k)}\in b$ since $a_{j}^{(k)}$
  is the $\P$-coreduction of $a_{j+1}^{(k)}$. We therefore have that $\beta_i\in\acl(b)$.
  
\end{proof}

\section{A Construction in $\DCF$}\label{sec5}

\noindent In this section we show that in $\DCF$ we can do better than the conclusions of Proposition
\ref{seqred}. Given any sequence of positive integers we provide a type which has a canonical $\C$-analysis with that $U$-type. Throughout we use the fact proven in Lemma \ref{dcf} that
  any finite rank type has a $\C$-coreduction.

Suppose $n_1,...,n_\ell$ are positive integers. We want to construct a type admitting a $\C$-analysis
in $\ell$ steps where the $i$th step has $U$-rank $n_i$, and such that the analysis is canonical.
Here is our construction.

  For convenience, we name everything in $\Q^{\alg}$ in the language.
  Let $c_{ij}\in\Q^{\alg}$ be algebraic numbers for $i=1,2,...,\ell$
  and $1\leq j\leq n_i$ such that $\{c_{ij}\}_{j=1}^{n_i}$ is $\Q$-linearly independent
  for $i=1,2,...,\ell$.

  We inductively define $(D_i,e_i)$ for $i=1,2,...,\ell$ as follows:

  Set $D_1:=\delta$ and let $e_1$ be a generic solution over $\varnothing$ to
  \begin{align}(D_1-c_{11})(D_1-c_{12})...(D_1-c_{1n_1})x=0. \tag{E$_1$}\end{align}
  For $i>1$ set $\displaystyle D_i:=\frac {\delta}{\prod_{j=1}^{i-1} e_{j}}$ and let $e_i$ be a generic solution over
  $\{e_1,...e_{i-1}\}$ to
  \begin{align}(D_i-c_{i1})(D_i-c_{i2})...(D_i-c_{in_i})x=0. \tag{E$_i$}\end{align}

   The notation $D_i-c_{ij}$ here represents a linear operator which sends $y$ to $D_iy-c_{ij}y$, so
  equation (E$_i$) is a linear differential equation over $\{e_1,...e_{i-1}\}$ of order $n_i$.

  Now let $a_i=(e_1,...,e_i)$ for $i=1,2,...,n$, and $a_0=\varnothing$.
   We will show that $(a_1...a_\ell)$ is a canonical $\C$-analysis of $a_\ell$ of $U$-type $(n_1,...,n_\ell)$.

  Since $e_i$ is a generic solution of (E$_i$), an order $n_i$ linear differential equation over $a_{i-1}$,
  we have $U(a_i/a_{i-1})=n_i$, and $\stp(a_i/a_{i-1})$ is almost $\C$-internal.
  So this is a $\C$-analysis
   of the correct $U$-type. We need to show it is by $\C$-reductions and $\C$-coreduction.

  Fixing $i\in\{1,2,...,\ell\}$, the following coordinatisation of solutions of (E$_i$) is a useful tool that we will apply often.

  \begin{lemma}\label{lclaim1}
  If $f$ is any solution to (E$_i$) then we can decompose $\displaystyle f=\sum_{j=1}^{n_i}f_{j}$
  such that each $f_j$ is a solution to $D_ix-c_{ij}x=0$ and $f$ is interdefinable with $(f_1,...,f_{n_i})$
  over $a_{i-1}$.
  \end{lemma}

  \begin{proof}
  Indeed, let $g_{j}$ be a generic solution of $D_ix-c_{ij}x=0$.
  The set $\{g_{j}:j=1,2,...,n_i\}$ is
  $\C$-linearly independent because $g_{j}$'s are nonzero eigenvectors of different eigenvalues under the $\C$-linear
  operator $D_i$. Note that since $(D_i-c_{ij})$ commutes with $(D_i-c_{ij'})$ for any $j,j'$, each $g_{j}$
  is a solution to (E$_i$). Since (E$_i$) is an order $n_i$ linear differential equation and $\{g_{j}:j=1,2,...,n_i\}$
  is a set of $\C$-linearly independent solutions of (E$_i$), any solution of (E$_i$)
  is a
  $\C$-linear combination of $g_{j}$'s. In particular, $f$ is of the form $\displaystyle\sum_{j=1}^{n_i}u_{j}g_{j}$
  where $u_{j}\in\C$ for $j=1,...,n_i$.
  Let $f_{j}=u_{j}g_{j}$, so $\displaystyle f=\sum_{j=1}^{n_i}f_{j}$, and $f\in\dcl(f_1,...,f_{n_i})$. Also,
  $$D_if_j-c_{ij}f_j=u_j(D_ig_j-c_{ij}g_j)=0,$$
  so $f_j$ is a solution to $D_ix-c_{ij}x=0$.

  We still need to verify that $(f_1,...,f_{n_i})\in\dcl(a_{i-1}f)$.
  Indeed, suppose $(f^*_j)_{j=1}^{n_i}$ and $(f_j)_{j=1}^{n_i}$ have the same type over $a_{i-1}f$.
  Then in particular $f^*_j$ is a solution to $D_ix-c_{ij}x=0$, and
  $$\sum_{j=1}^{n_i}f_{j}=f=\sum_{j=1}^{n_i}f^*_{j}$$
  which gives us $\displaystyle\sum_{j=1}^{n_i}(f_{j}-f^*_j)=0$. As $\{f_{j}-f^*_j:j=1,2,...,n_i\}$ is a set of
  eigenvectors of different eigenvalues under the $\C$-linear
  operator $D_i$, we then have $f_{j}-f^*_j=0$ for all $j=1,2,...,n_i$, so $(f^*_j)_{j=1}^{n_i}=(f_j)_{j=1}^{n_i}$.
  \end{proof}

  \begin{lemma}\label{lclaim2} If $f$ is a generic solution to (E$_i$) over $a_{i-1}$,
  then $\{f_1,...,f_{n_i}\}$ obtained in Lemma \ref{lclaim1}
  is independent over $a_{i-1}$ and each $f_j$ is a generic solution to $D_ix-c_{ij}x=0$.
  \end{lemma}

  \begin{proof} Since $f$ is a generic solution over $a_{i-1}$ to (E$_i$), which is a
  linear differential equation of order $n_i$, we have $U(f/a_{i-1})=n_i$
  Since $f_j$ is a solution for $D_ix-c_{ij}x=0$,
  $U(f_{ij}/a_{i-1})\leq 1$. But
  \begin{equation*}\begin{split}
  n_i=&U(f/a_{i-1})\\
  =&U(f_1f_2...f_{n_i}/a_{i-1})\\
  =&U(f_1/a_{i-1})+U(f_2/a_{i-1}f_1)+...+U(f_{n_i}/a_{i-1}f_1f_2...f_{n_i-1})\\
  \leq &U(f_1/a_{i-1})+U(f_2/a_{i-1})+...+U(f_{n_i}/a_{i-1})\\
  \leq&n_i.
  \end{split}\end{equation*}
  So $U(f_j/a_{i-1})=1$ and $U(f_{j}/a_{i-1}f_1f_2...f_{j-1})=1$ for $j=1,2,...,n_i$.
  This means that $\{f_1,...,f_{n_i}\}$
  is independent over $a_{i-1}$ and each $f_j$ is a generic solution to $D_ix-c_{ij}x=0$.
  \end{proof}

  \begin{lemma}\label{lclaim3}
  Let $f$ be a generic solution over $\Q^{\alg}$ to (E$_1$). Then $\acl(f)\cap\C=\Q^{\alg}$.
  \end{lemma}

  \begin{proof}
    Let $m=n_1$. Let $(f_1,...,f_m)$ be the decomposition of $f$ by Lemma \ref{lclaim1} with respect to (E$_1$).
    Since $f$ is generic, $f_j\neq 0$ for $j=1,2,...,m$.
    Suppose the conclusion is false and there exists some $c$ such that $c\in(\acl(f)\cap\C)\backslash\Q^{\alg}$.
    Note that $\acl(f)=\Q( f_1,...,f_m)^{\alg}$ since $\delta f_j=c_{1j}f_j\in\Q^{\alg}(f_j)$.

  For simplicity, let $\bar f=(f_1,...,f_m)$, and $\bar y=(y_1,...,y_m)$.
  Let $F(x,\bar y)$ be a polynomial with coefficients in $\Q^{\alg}$ such that
  $F(c,\bar f)=0$ and $F(x,\bar f)\neq0$.
  Since $c\not\in\Q^{alg}$, $F(c,\bar y)\neq 0$.
  Let $G(\bar y)$ be a nonzero polynomial over $\C$ with minimal number of terms such that
  $G(\bar f)=0$. Since $F(c,\bar y)\neq 0$ and $F(c,\bar f)=0$, $F(c,\bar y)$ satisfies all conditions of
  $G$ except for the minimality, so such a $G$ exists.

  Let
  $$G(\bar y)=\sum_{\bar r\in I}s_{\bar r}\bar y^{\bar r},$$
  where $I$ is a set of $m$-tuples of
  nonnegative integers, $\bar y^{\bar r}=y_1^{r_1}...y_m^{r_m}$, and $s_{\bar r}\in\C$.
  Let $\bar c=(c_{11},...,c_{1m})$,
  and set $\displaystyle \bar f\bar c:=\sum_{j=1}^mf_jc_{1j}$.

  We claim that
    $$\bar r^{(1)}\bar c=\bar r^{(2)}\bar c$$
  for all $\bar r^{(1)},\bar r^{(2)}\in I$. Indeed, otherwise, fixing any $\bar r^*\in I$, we have
  \begin{equation*}\begin{split}
    G^*(\bar y):=&\bar r^*\bar cG(\bar y)-\delta(G(\bar y))\\
    =&\sum_{\bar r\in I}(\bar r^*\bar c)s_{\bar r}\bar y^{\bar r}-\sum_{\bar r\in I} s_{\bar r}\delta\bar y^{\bar r}\\
    =&\sum_{\bar r\in I}(\bar r^*\bar c-\bar r \bar c)s_{\bar r}\bar y^{\bar r}
  \end{split}\end{equation*}
  is a polynomial with fewer terms than $G$ (since the term with index $\bar r^*$ is cancelled) such that its
  coefficients are in $\C$, $G^*(\bar f)=0$ since $G(\bar f)=\delta(G(\bar f))=0$, and $G^*(\bar y)\neq 0$
  as there exist $\bar r\in I$ such that
  $\bar r\bar c\neq\bar r^{(*)}\bar c$. This contradicts the minimality of $G$.

  We now have $\bar r^{(1)}\bar c=\bar r^{(2)}\bar c$ for all $\bar r^{(1)},\bar r^{(2)}\in I$,
   i.e., $(\bar r^{(1)}-\bar r^{(2)})\bar c=0$ for all $\bar r^{(1)},\bar r^{(2)}\in I$.
   But $\{c_{11}...,c_{1m}\}$ is $\Q$-linearly independent, so in fact $\bar r^{(1)}=\bar r^{(2)}$
   for all $\bar r^{(1)},\bar r^{(2)}\in I$.
  Therefore there is only one element $\bar r$ in $I$, and
   $G(\bar f)=s_{\bar r}\bar f^{\bar r}$. Since all $f_j$'s are nonzero, $s_{\bar r}=0$, so $G$ is the zero
   polynomial, a contradiction.
   \end{proof}

  The following lemma is the technical heart of the construction.

  \begin{lemma}\label{lclaim4}
   Fix $i\in\{1,2,...,\ell-1\}$, and for notational convenience, let $m:=n_i$ and $L:=\acl(a_{i-1})$.
  Then the following are true:
  \begin{enumerate}[label=(\roman*)]
  \item
  Suppose $f$ is a solution of (E$_i$) and $(f_1,...,f_m)$ is the decomposition of $f$ by Lemma \ref{lclaim1}.
  Then $f$ is generic over $L$ iff all the $f_j$ are nonzero.
  \item Suppose $f$ is a generic solution to (E$_i$) over $L$, $\alpha\in\Q^{\alg}$ is nonzero,
  and $h$ is a nonzero solution of $D_ix-\alpha fx=0$. Then
    $f$ is the $\C$-coreduction of $h$ over $L$.
  \item
    The $\C$-coreduction of $a_{i+1}$ over $a_{i-1}$ is $a_{i}$.
  \item
    The $\C$-reduction of $a_{i+1}$ over $a_{i-1}$ is $a_{i}$.
  \end{enumerate}
  \end{lemma}

\begin{proof} We use induction on $i$.

  (i) Suppose the conclusion is true for $i-1$.

  By Lemma \ref{lclaim2}, if $f$ is a generic solution to (E$_i$) over $L$, then for any $j\in\{1,2,...,m\}$, $f_j$
  is a generic solution to $D_ix-c_{ij}x=0$. In particular, $f_j\neq0$.

  Now suppose $f_j\neq 0$ for all $j=1,2,...,m$, but $f$ is not generic, i.e., $U(f/L)<m$.
  Since
  \begin{equation*}\begin{split}
  U(f/a_{i-1})=&U(f_1f_2...f_{m}/a_{i-1})\\
  =&U(f_1/a_{i-1})+U(f_2/a_{i-1}f_1)+...+U(f_{m}/a_{i-1}f_1f_2...f_{m-1}),\\
  \end{split}\end{equation*}
  $U(f_j/a_{i-1}f_1f_2...f_{j-1})<1$ for some $j$, and hence
   $\displaystyle f_j\in L\langle \bigcup_{k\neq j} f_k\rangle$ for that $j$.
   Note that
   $$\delta f_k=(D_if_k)\prod_{j=1}^{i-1} e_{j}=c_{ik}f_k\prod_{j=1}^{i-1} e_{j}\in L(f_k),$$
   so $f_j\in L\langle \bigcup_{k\neq j} f_k\rangle=L(\bigcup_{k\neq j} f_k)$,
   which means that $\{f_1,...,f_m\}$ is algebraically dependent over $L$ in the field theoretic sense.

  Let $\bar f=(f_{1},...,f_{m})$, $\bar y=(y_{1},...,y_{m})$, and $\bar c=(c_{i1},...,c_{im})$.
  Let $G(\bar y)$ be a nonzero polynomial with minimal number of terms such that its coefficients are in $L$ and
  $G(\bar f)=0$. We will use a minimality argument similar to that in the proof of Lemma \ref{lclaim3}.
  Suppose
   $$G(y_1,...,y_{m})=\sum_{\bar r\in I}s_{\bar r}\bar y^{\bar r},$$
  where $I$ is a set of $m$-tuples of
  nonnegative integers, and $s_{\bar r}\in L$ for $\bar r\in I$.
  Now
    \begin{equation*}\begin{split}
    D_i(G(\bar f))=&D_i\sum_{\bar r\in I}s_{\bar r}\bar f^{\bar r}\\
      =&\sum_{\bar r\in I}(\bar f^{\bar r}D_is_{\bar r}+s_{\bar r}D_i\bar f^{\bar r})\\
      =&\sum_{\bar r\in I}(\log D_i s_{\bar r}+\bar r\bar c)s_{\bar r}\bar f^{\bar r}.
    \end{split}\end{equation*}

  We claim that
  $$\log D_i s_{\bar r^{(1)}}+\bar r^{(1)}\bar c=\log D_i s_{\bar r^{(2)}}+\bar r^{(2)}\bar c$$
  for all $\bar r^{(1)},\bar r^{(2)}\in I$. Indeed, otherwise, fixing any $\bar r^*\in I$, we have
    \begin{equation*}\begin{split}
    G^*(\bar y):=&(\log D_i s_{\bar r^{*}}+\bar r^{*}\bar c)G(\bar y)-D_i(G(\bar y))\\
                =&\sum_{\bar r\in I}(\log D_i s_{\bar r^{*}}+\bar r^{*}\bar c
                -\log D_i s_{\bar r}-\bar r\bar c)s_{\bar r}\bar y^{\bar r}
    \end{split}\end{equation*}
  is a polynomial with fewer terms than $G$ (since the term with index $\bar r^*$ is cancelled)
  such that its coefficients are in $L$, $G^*(\bar f)=0$ as $G(\bar f)=D_i(G(\bar f))=0$,
  and $G^*(\bar y)\neq 0$ as there exist $\bar r$ in $I$ such that
  $\log D_i s_{\bar r}+\bar r\bar c\neq \log D_i s_{\bar r^{*}}+\bar r^{*}\bar c$.
  This contradicts the minimality of $G$.

  There are at least two terms in $G(\bar y)$. Indeed, if there is only one term in $G$,
  then $G(\bar y)=s_{\bar r}\bar y^{\bar r}$ for the unique $\bar r\in I$. Since $f_j\neq 0$ for
  $j=1,2,...,m$ and $G(\bar f)=0$, we have $s_{\bar r}=0$, so $G(\bar y)=0$, contradicting the fact
  that $G$ is nonzero.

  We now have
  $$\log D_i s_{\bar r^{(1)}}+\bar r^{(1)}\bar c=\log D_i s_{\bar r^{(2)}}+\bar r^{(2)}\bar c$$
  for all $\bar r^{(1)},\bar r^{(2)}\in I$. Note that
   $\log D_i s_{\bar r}+\bar r\bar c=\log D_i (s_{\bar r}\bar f^{\bar r})$
  for any $\bar r\in I$. Therefore, fixing $\bar r^{(1)}\neq\bar r^{(2)}$ in $I$, we get
  $s_{\bar r^{(1)}}\bar f^{\bar r^{(1)}}=cs_{\bar r^{(2)}}\bar f^{\bar r^{(2)}}$
  for some nonzero $c\in\C$. This means that
    \begin{align}c\bar f^{\bar r^{(2)}-\bar r^{(1)}}=s_{\bar r^{(1)}}s_{\bar r^{(2)}}^{-1}.\tag{*}\end{align}
  Note that as all $f_j\neq0$, $\bar f^{\bar r^{(2)}-\bar r^{(1)}}$ makes sense and is nonzero.
  Let $h=c\bar f^{\bar r^{(2)}-\bar r^{(1)}}$.
  Then $h$ is a nonzero solution to
    \begin{align}\log D_i x=(\bar r^{(2)}-\bar r^{(1)})\bar c.\tag{**}\end{align}
  When $i=1$, right side of (*) is in $\acl(a_0)=\Q^{\alg}\subset \C$,
   so $h$ is also a constant,
  but then it is not a solution for (**). When $i>1$, we apply part (ii) of the lemma
  for $i-1$ with $e_{i-1}$ a generic solution of (E$_{i-1}$) over $a_{i-2}$,
  $\alpha=(\bar r^{(2)}-\bar r^{(1)})\bar c^*\neq0$, and $h$
  a nonzero solution of $D_{i-1}x-dx=0$.
   We get that $e_{i-1}$ is the coreduction of $h$ over $a_{i-2}$.
   In particular, since $e_{i-1}\not\in\acl(a_{i-2})$, we have that $\stp(h/a_{i-2})$ is not almost $\C$-internal.
  But the right side of (*)
  is in $L$ which is almost $\C$-internal over $a_{i-2}$, a contradiction.

  (ii) Suppose the conclusion is true for $i-1$, and (i) is true for $i$.

  We use induction on $m$, the order of the differential equation (E$_i$).

  If $m=n_i=1$, we have that $\log D_ih=\alpha f$ and $\log D_i(\alpha f)=c_{i1}$.
  Let $h^*$ be a generic solution of $\log D_ix=\alpha f$ over $Lf$. Since $f$ is a generic solution of
    $\log D_i(x)=c_{i1}$ over $L$, $\alpha f$ is also a generic solution of
    $\log D_i(x)=c_{i1}$ over $L$, and therefore $h^{*}$ is a generic solution of $\log D_i^{(2)}x=c_{i1}$
    over $a_{i-1}$. Thus $\stp(h^*/L)$ is not almost $\C$-internal by Proposition \ref{fprop1}.
    Since $h^*$ is a constant multiple of $h$, $\stp(h/L)$ is also not almost $\C$-internal.
    Note that $(f,h)$ is a $\C$-analysis of $h$ over $L$, and as
    it is incompressible of $U$-type $(1,1)$, we have that $f$ is the $\C$-coreduction of $h$
    over $L$.

  Now suppose the conclusion of (ii) is proven if the order of the equation (E$_i$) is less than or equal to $m-1$.

  Let $\beta$ be the $\C$-coreduction of $h$ over $L$.
  Since $\stp(h/Lf)$ is almost $\C$-internal, we only need to show that $f\in\acl({{L}}\beta)$.
  Let $(f_{1},...,f_{m})$ be the decomposition of $f$ by Lemma \ref{lclaim1}. By Lemma \ref{lclaim2},
   $f_j$ is a generic solution
  of $D_ix-c_{ij}x=0$ for $j=1,2,...,m$.
  Suppose towards a contradiction that $f\not\in\acl(L\beta)$.
   We may, without loss of generality, suppose $f_{1},...,f_{s}\not\in\acl(L\beta)$ and
   $f_{s+1},...,f_{m}\in\acl(L\beta)$ for some $1\leq s\leq m$.

  In the rest of the proof we seek a contradiction to the above assumption.

  We prove first that $s=m$. Suppose not, so $f_m\in \acl(L\beta)$.
  Let $h_m$ be a nonzero solution to $D_ix-\alpha f_mx=0$. We have that $\stp(h_m/Lf_m)$
  is almost $\C$-internal. Since $f_m\in\acl(L\beta)$, $\stp(h_m/L\beta)$ is almost $\C$-internal.
  Let $h^*=hh_m^{-1}$. Then 
  \begin{equation*}\begin{split}
    \log D_i( h^*)&=\log D_i (h)-\log D_i(h_m)\\
    &=\alpha (f_1+...+f_{m-1}+f_m)-\alpha f_m\\
    &=\alpha (f_1+...+f_{m-1}).
  \end{split}\end{equation*}
  Let $f^*=f_1+...+f_{m-1}$. Then $h^*$ is a nonzero solution to
  $D_ix-\alpha f^*x=0.$
   From (i), since $f_1,...,f_{m-1}$ are all nonzero, $f^*$
  is a generic solution over $L$ to
  $$(D_i-c_{i1})...(D_i-c_{i,m-1})x=0.$$
  By the induction hypothesis,
  we conclude that the $\C$-coreduction of $h^*$ over $L$ is
  $f^*$. Since $h$ and $h_m$ are almost $\C$-internal over $L\beta$ and $h^*=hh_m^{-1}$,
  we get that $f^*\in\acl(L\beta)$.
  As $f^*$ is
  interdefinable with $(f_1,...,f_{m-1})$ over $L$, $f_1\in\acl({L}\beta)$,
  contradicting our assumption.

  Let $g_{t1}=tf_1$ for $t=1,2,...$. We show that $\stp(g_{t1}/L\beta)=\stp(f_1/L\beta)$.
  Since
    \begin{equation}\label{claim3}D_ig_{t1}-c_{i1}g_{t1}=tD_if_1-tc_{i1}f_1=0,\end{equation}
  we have that $g_{t1}\in\{x:D_ix-c_{i1}x=0\}$, a strongly minimal set.
  Thus
  in order to prove $\stp(g_{t1}/L\beta)=\stp(f_1/L\beta)$ we only need to show that
  $g_{t1}\not\in\acl(L\beta)$, which follows from $f_1\not\in\acl(L\beta)$.

  For each integer $t\geq 1$, let $\eta_t$ be an automorphism
  fixing $\acl(L\beta)$ and taking $f_1$ to $g_{t1}$.
  Set $g_{tj}:=\eta_t(f_j)$ for all $j=1,2,...,m$, $g_t:=\eta_t(f)$, and $h_t:=\eta_t(h)$.
  So $\stp(h_t,g_t,g_{t1},...,g_{tm}/L\beta)=\stp(h,f,f_{1},...,f_{m}/L\beta)$ for all $t\geq 1$. In particular,
  $g_t$ is a generic solution to (E$_i$) over $L$, $h_t$ is a nonzero solution to
  $D_ix-\alpha g_tx=0$, $\displaystyle g_t=\sum_{j=1}^mg_{tj}$ is the decomposition by Lemma \ref{lclaim1},
  and $\stp(h_t/\beta)$ is almost $\C$-internal.

  We next show that $g_{tj}=tf_{j}$ for all $t\geq1$ and all $j$.

  Towards a contradiction,
  suppose that $g_{tj}\neq tf_{j}$ for some $t$ and $j$.
  Fix this $t$. We argue first that $g_{tj}-tf_j\in\acl(L\beta)$. Let $H=h_th^{-t}$, and
  let $I=\{j:2\leq j\leq m,g_{tj}-tf_{j}\neq0\}$ (note that $g_{t1}=tf_{1}$, so
  we only need $j\geq 2$; also note that $I$ is nonempty since $g_{tj}\neq tf_{j}$ for some $j$
  by assumption).
  We have that
  \begin{equation*}\begin{split}
    D_i H=&(\log D_iH)H\\
    =&(\log D_ih_t-t\log D_ih)H\\
    =&(\alpha g_t-t\alpha f)H\\
    =&(\alpha \sum_{j=1}^m(g_{tj}-tf_{j}))H,\\
    =&(\alpha \sum_{j\in I}(g_{tj}-tf_{j}))H.
  \end{split}\end{equation*}
  So $H$ is a nonzero solution of $D_ix-(\alpha\sum_{j\in I}(g_{tj}-tf_{j}))x=0$.

  Note that $\sum_{j\in I}(g_{tj}-tf_{j})$ is a solution to
  \begin{equation}\label{claim32}\left(\prod_{j\in I} (D_i-c_{ij})\right)(x)=0.\end{equation}
  This is because (\ref{claim32}) is linear, and for each $j\in I$,
  $$(D_i-c_{ij})(g_{tj}-tf_{j})=(D_i-c_{ij})g_{tj}-(D_i-c_{ij})tf_{j}=0.$$

  The decomposition of $\displaystyle\sum_{j\in I}(g_{tj}-tf_{j})$ by Lemma \ref{lclaim1} with respect to (\ref{claim32}) is
  $(g_{tj}-tf_{j})_{j\in I}$, and $g_{tj}-tf_{j}\neq 0$ for every $j\in I$. Therefore,
  applying part (i) where we replace (E$_i$) with (\ref{claim32}),
  we get that $\sum_{j\in I}(g_{tj}-tf_{j})$ is a generic solution to (\ref{claim32}) over $L$.

  Now, since (\ref{claim32}) is of order less than $m$
  and $H$ is a nonzero solution of $D_ix-(\alpha\sum_{j\in I}(g_{tj}-tf_{j}))x=0$,
  by the induction hypothesis, the coreduction of $H$ over $L$ is $\sum_{j\in I}(g_{tj}-tf_{j})$.
  Since $H=h_th^{-t}$ and both $h$ and $h_t$ are almost $\C$-internal over
  $L\beta$, we have $\stp(H/L\beta)$
  is almost $\C$-internal. Therefore, for any $j\in I$, $g_{tj}-tf_{j}\in\acl(L\beta)$.
  We now fix some $j\in I$.

  Let $\gamma=\displaystyle\frac{g_{tj}}{f_{j}}-t=\displaystyle\frac{g_{tj}-tf_{j}}{f_{j}}\neq 0$.
  Then $\gamma$ is a constant in $\acl(LF)\backslash\acl(L\beta)$. Indeed, $\gamma$ is a constant
  because $g_{tj}$ and $f_j$ are both solutions to $D_ix-c_{ij}x=0$, and hence
  $\displaystyle\frac{g_{tj}}{f_{j}}\in\C$.
  We get $\gamma\in\acl(Lf)$ by the fact that $g_{tj}-tf_j\in\acl(L\beta)\subseteq\acl(Lf)$.
  And $\gamma\not\in\acl(L\beta)$ because if it were, then so would
  $\displaystyle f_j=\frac{g_{tj}-tf_{j}}{\gamma}$,
  but we know that is not the case.

  When $i=1$ this is impossible, since $\acl(Lf)=\acl(f)$, and Lemma \ref{lclaim3} tells us that
  $\acl(f)\cap\C=\Q^{\alg}$.

  Suppose $i>1$.  We apply part (iv) of the lemma for $i-1$
  and get that the $\C$-reduction
  of $a_i$ over $a_{i-2}$ is $a_{i-1}$.
  As $f$ is a generic solution of (E$_i$) over $L$,
  $\stp(f/L)=\stp(e_i/L)$, so the $\C$-reduction
  of $f$ over $a_{i-2}$ is $a_{i-1}$.
  Since $\gamma\in\acl(Lf)\backslash\acl(L\beta)$,
  $\gamma\not\in L=\acl(a_{i-1})$. So
  $\stp(\gamma/a_{i-2})$ is not almost $\C$-internal. On the other hand, $\gamma$ is a constant, a contradiction.

  What we have actually shown is that for any $t\geq 1$, $\stp(tf_1/L\beta)=\stp(f_1/L\beta)$, and if
  $\stp(\tilde f_1,\tilde f_2,...,\tilde f_m/L\beta)=\stp(f_1,...,f_m/L\beta)$ and
  $\tilde f_1=tf_1$, then $\tilde f_j=tf_j$ for $j=2,3,...,m$. In particular,  $\stp(tf_1,...,tf_m/L\beta)=\stp(f_1,...,f_m/L\beta)$ holds
  for all $t$. In addition, the case of $t=1$ tells us that $f_j\in\dcl(f_1\acl(L\beta))$
  for $j=2,3,...,m$.

  We now show that $\displaystyle\frac{f_j}{f_1}\in\acl(L\beta)$ for $j=2,3,...,m$.
  Fix some $j$.
  Since $f_j\in\dcl(f_1\acl(L\beta))$, there exists a formula $\varphi_1(x,y)$ over $\acl(L\beta)$
  such that
  $\varphi_1(\U,f_1)=\{f_j\}$. Since $\stp(tf_1,tf_j/L\beta)=\stp(f_1,f_j/L\beta)$,
  we have $\varphi_1(\U,tf_1)=\{tf_j\}$ for all $t$. Now set
  $\displaystyle\varphi_2(x,y):=\forall z(\varphi_1(z,y)\rightarrow x=\frac zy)$.
  Then $\displaystyle\varphi_2(\U,tf_1)=\left\{\frac{f_j}{f_1}\right\}$ for all $t$.
  So we have
  $$\{tf_1:t\geq1\}\subseteq\left\{b\in\U:\log D_ib=c_{i1}\text{ and }\varphi_2(\U,b)=\left\{\frac{f_j}{f_1}\right\}\right\}.$$
  Since $\log D_ix=c_{i1}$ is strongly minimal, it must be that for all but finitely many solutions
  to $\log D_ix=c_{i1}$, $\displaystyle\varphi_2(\U,b)=\left\{\frac{f_j}{f_1}\right\}$.
  It follows that $\displaystyle\frac{f_j}{f_1}\in\acl(L\beta)$.

  Let $g_{01}$ be a generic solution over
  $Lh$ to $D_ix-c_{i1}x=0$, and $\displaystyle g_{0j}=g_{01}\frac{f_{j}}{f_{1}}$ for $j=2,3,...,m$.
  We have shown that each $\displaystyle\frac {f_{j}}{f_{1}}$
  is in $\acl(L\beta)$, so $(g_{01},...,g_{0m})\in\acl(L\beta g_{01})$.
  Let $\displaystyle c_{01}=\frac{f_{1}}{g_{01}}\in\C$.
  Now,
  \begin{equation*}\begin{split}
  \log D_i^{(2)}(h)=&\log D_i(\alpha f)\\
  =&\log D_i(\alpha(f_1+...+f_m))\\
  =&\log D_i(\alpha c_{01}(g_{01}+...+g_{0m}))\\
  =&\log D_i(g_{01}+...+g_{0m})=:\epsilon.
  \end{split}\end{equation*}
  Hence $h$ is a solution to $\log D_i^{(2)}(x)=\epsilon$ which is over $\acl(L\beta g_{01})$,
  so $U(h/L\beta g_{01})\leq 2$.
  Note that $U(h/L\beta)\geq 2$ since $h$ is a generic solution to $\log D_i x=\alpha f$ and $U(f/L\beta)\geq 1$.
  But we also have $h\underset{L\beta}\forkindep g_{01}$ (recall that $\beta\in\acl(Lh)$),
  so $U(h/L\beta g_{01})=U(h/L\beta)\geq 2$. Thus $U(h/L\beta g_{01})=2$, and $h$ is a generic solution to
  $\log D_i^{(2)}(x)=\epsilon$ over $\acl(L\beta g_{01})$.
  Hence $\stp(h/L\beta g_{01})$ is not almost $\C$-internal by Proposition \ref{fprop1},
  and therefore $\stp(h/L\beta)$ is not almost $\C$-internal, contradicting the definition of $\beta$.

  (iii)
  Assume part (ii) of the lemma is true for $i$.

  Let $e_{i+1}=\sum_{j=1}^{n_{i+1}}b_{i+1,j}$ be the decomposition by Lemma \ref{lclaim1} with respect to (E$_{i+1}$).
   We have that $\stp(a_{i+1}/a_i)$ is almost $\C$-internal. Also,
  by part (ii) applied to $f=e_i$ and $h=b_{i+1,1}$, the $\C$-coreduction of $b_{i+1,1}$ over $a_{i-1}$ is $e_i$,
  which is interdefinable over $a_{i-1}$ with $a_i$. Since $b_{i+1,1}\in\dcl(a_ie_{i+1})=\dcl(a_{i+1})$,
  the $\C$-coreduction of $a_{i+1}$ over $a_{i-1}$ is $a_i$.

  (iv)
  Assume parts (i) and (ii) of the lemma are true for $i$. For simplicity, we use $n$ to denote $n_{i+1}$. Let $K$ be the algebraically closed field generated by $a_{i}$. Let $\bar b_{i+1}=(b_{i+1,1},...,b_{i+1,n})$.

  We already know that $\stp(a_i/a_{i-1})$ is $\C$-internal.
  Suppose $\beta\in\acl(a_{i+1})$ is almost $\C$-internal over $a_{i-1}$ and $\beta\not\in\acl(a_i)$.
  Since $e_{i+1}$ is interalgebraic with $\bar b_{i+1}$ over $a_{i}$,
  $\beta\in\acl(a_{i}\bar b_{i+1})$, which means
  $\beta\in K\langle \bar b_{i+1}\rangle^{\alg}$. Since $\displaystyle \delta b_{i+1,j}=c_{i+1,j}b_{i+1,j}\prod_{k=1}^{i} e_{k}
  \in K(b_{i+1,j})$ for $j=1,2,...,n$, we have $K\langle \bar b_{i+1}\rangle=K(\bar b_{i+1})$,
  so $\beta\in K( \bar b_{i+1})^{\alg}$.
  Thus there exist a polynomial $F(x,y_1,...,y_{n})$ with coefficients in $K$
  such that $F(\beta,b_{i+1,1},...,b_{i+1,n})=0$ and $F(x,b_{i+1,1},...,b_{i+1,n})\neq0$.
  Also, $F(\beta,y_1,...,y_n)\neq 0$ since $\beta\not\in K$.

  Suppose $G(y_1,...,y_{n})$
  is a nonzero polynomial with minimal number of terms such that the coefficients of $G$ are almost $\C$-internal
  over $a_{i-1}$
  and $G(\bar b_{i+1})=0$. Note that this is well-defined because $F(\beta,y_1,...,y_n)$ satisfies
  all the conditions except for the minimality, as $K$ and $\beta$ are both almost $\C$-internal over $a_{i-1}$.

  Let
    $$G(y_1,...,y_{n})=\sum_{\bar r\in I}s_{\bar r}\bar y^{\bar r},$$
  where $I$ is a set of $n$-tuples of
  nonnegative integers, and $\stp(s_{\bar r}/a_{i-1})$ is almost $\C$-internal.
  Let $\bar c_{i+1}=(c_{i+1,1},...,c_{i+1,n})$.
  Arguing exactly as in the proof of part (i) of the lemma, we get by minimality of $G$ that
  \begin{equation}\label{claim34}\log D_i s_{\bar r^{(1)}}+\bar r^{(1)}\bar c_{i+1}e_i=
  \log D_i s_{\bar r^{(2)}}+\bar r^{(2)}\bar c_{i+1}e_i
  \end{equation}
  for any $r^{(1)},r^{(2)}\in I$. Indeed,
  \begin{equation*}\begin{split}
  D_i(G(\bar b_{i+1}))
  =&\sum_{\bar \in I} (\bar b_{i+1}^{\bar r} D_is_{\bar r}+s_{\bar r}D_i\bar b_{i+1}^{\bar r})\\
  =&\sum_{\bar \in I} (\bar b_{i+1}^{\bar r} D_is_{\bar r}+s_{\bar r}\bar r\bar c_{i+1} e_i\bar b_{i+1}^{\bar r})\\
  =&\sum_{\bar \in I} ( \log D_is_{\bar r}+\bar r\bar c_{i+1} e_i)s_{\bar r}\bar b_{i+1}^{\bar r},\\
  \end{split}\end{equation*}
  where the second equality is by the fact that
  \begin{equation*}\begin{split}
    D_i\bar b_{i+1}^{\bar r}=&\bar r\bar b_{i+1}^{\bar r-\bar 1}D_i\bar b_{i+1}\\
    =&\bar r\bar b_{i+1}^{\bar r-\bar 1}e_iD_{i+1}\bar b_{i+1}\\
    =&\bar r\bar b_{i+1}^{\bar r-\bar 1}e_i\bar c_{i+1}\bar b_{i+1}\\
    =&\bar re_i\bar c_{i+1}\bar b_{i+1}^{\bar r}.\\
  \end{split}\end{equation*}
  Now if (\ref{claim34}) failed, then fixing any $\bar r^*\in I$ we see that
  \begin{equation*}\begin{split}
    G^*(\bar y):=&(\log D_i s_{\bar r^*}+\bar r^*\bar c_{i+1}e_i)G(\bar y)-D_iG(\bar y)\\
    =&\sum_{\bar r\in I}(\log D_i s_{\bar r^*}+\bar r^*\bar c_{i+1}e_i
    -\log D_i s_{\bar r}-\bar r\bar c_{i+1}e_i)s_{\bar r}\bar y^{\bar r}
  \end{split}\end{equation*}
  whose coefficients are again almost $\C$-internal over $a_{i-1}$, would contradict the minimal choice of $G$.

  If $G$ has only one term, then for the only $\bar r\in I$, $G(\bar b_{i+1})=s_{\bar r}\bar b_{i+1}^{\bar r}$.
  Since $b_{i+1,j}\neq 0$ for $j=1,2,...,n$, $s_{\bar r}=0$, which means $G(\bar y)=0$,
  a contradiction. Now fix $r^{(1)}\neq r^{(2)}$ in $I$.
  Since $\log D_i s_{\bar r}+\bar r\bar c_{i+1}e_i=\log D_i (s_{\bar r}\bar b_{i+1}^{\bar r})$
  for any $\bar r\in I$, we have
  $s_{\bar r^{(1)}}\bar b_{i+1}^{\bar r^{(1)}}=cs_{\bar r^{(2)}}\bar b_{i+1}^{\bar r^{(2)}}$
  for some $c\in\C$. This means that
  $\bar b_{i+1}^{\bar r^{(1)}-\bar r^{(2)}}=cs_{\bar r^{(2)}}s_{\bar r^{(1)}}^{-1}$.
  So $\bar b_{i+1}^{\bar r^{(1)}-\bar r^{(2)}}$ is almost $\C$-internal over $a_{i-1}$.

  On the other hand, as $D_{i+1}\bar b_{i+1}^{\bar r^{(1)}-\bar r^{(2)}}=(\bar r^{(1)}-\bar r^{(2)})\bar c_{i+1}\bar b_{i+1}^{\bar r^{(1)}-\bar r^{(2)}}$,
  $\bar b_{i+1}^{\bar r^{(1)}-\bar r^{(2)}}$ is a solution of $(D_{i+1}-(\bar r^{(1)}-\bar r^{(2)})\bar c_{i+1})x=0$,
  with $(\bar r-\bar r^*)\bar c_{i+1}\neq 0$ since $\{c_{i+1,j}:j=1,2,...,n\}$ is $\Q$-linearly independent.
  By part (ii) of the lemma with $f=e_i$, $h=\bar b_{i+1}^{\bar r^{(1)}-\bar r^{(2)}}$, and
  $\alpha=(\bar r^{(1)}-\bar r^{(2)})\bar c_{i+1}$,
  $e_i$ is a $\C$-coreduction of $\bar b_{i+1}^{\bar r^{(1)}-\bar r^{(2)}}$
  over $a_{i-1}$. In particular,
  $\bar b_{i+1}^{\bar r^{(1)}-\bar r^{(2)}}$ is not almost $\C$-internal over $a_{i-1}$. This contradiction proves
  part (iv) of the lemma.
  \end{proof}

  We have accomplished the desired construction:

  \begin{theorem}
    Given positive integers $n_1,...,n_\ell$, there exists in $\DCF$ a type over $\Q^{\alg}$ that admits a canonical
    $\C$-analysis of $U$-type $(n_1,...,n_\ell)$.
  \end{theorem}

  \begin{proof}
  Let $(a_1,...,a_\ell)$ be as in the above construction. We have seen that $(a_1,...,a_\ell)$ is a
  $\C$-analysis of $p=\stp(a_\ell)$ of $\U$-type $(n_1,...,n_\ell)$.
  By Lemmas \ref{lemcr1} and \ref{lemcr2}, parts (iii) and (iv) of Lemma \ref{lclaim4}
  prove that it
  is a $\C$-analysis by reductions and coreductions, as desired.
  \end{proof}

\bibliographystyle{plain}
\bibliography{diffDME}

\end{document}